\documentclass[11pt]{amsart}%
\usepackage{amsmath,amssymb,latexsym,graphicx}
\usepackage{enumerate}
\usepackage{amsmath}
\usepackage{amsfonts}
\usepackage{amssymb}
\usepackage{color}
\usepackage{cite}
\usepackage{graphicx}%
\providecommand{\U}[1]{\protect\rule{.1in}{.1in}}
\newtheorem{theorem}{Theorem}[section]
\newtheorem{lemma}[theorem]{Lemma}

\newtheorem{remark}[theorem]{Remark}

\newtheorem{definition}[theorem]{Definition}

\definecolor{gree}{rgb}{0.,0.4,0}
\begin{document}
\title{An analog of polynomially integrable bodies in even-dimensional spaces}
\author{M. Agranovsky, A. Koldobsky, D. Ryabogin, and V. Yaskin}
\thanks{The second author was supported in part by the U.S. National Science Foundation Grant DMS-2054068. The third author was supported in part by the U.S. National Science Foundation Grant DMS-2000304. The fourth author was supported in part by NSERC}

\address[Mark Agranovsky]{Department of Mathematics, Bar-Ilan University, Ramat Gan, 5290002, Israel}

\email{agranovs@math.biu.ac.il}

\address[Alexander Koldobsky]{Department of Mathematics,
	University of Missouri-Columbia,
	Columbia, MO 65211, USA}

\email{koldobskiya@missouri.edu}

\address[Dmitry Ryabogin]{Department of Mathematics, Kent State University, Kent, OH 44242, USA}

\email{ryabogin@math.kent.edu}

\address[Vladyslav Yaskin]{Department of Mathematical and Statistical Sciences, University of Alberta, Edmonton, AB T6G 2G1, Canada}

\email{yaskin@ualberta.ca}

\begin{abstract}
 A bounded domain $K \subset \mathbb R^n$ is called polynomially integrable if
the $(n-1)$-dimensional volume of the intersection $K$ with a hyperplane $\Pi$ polynomially depends on the distance from $\Pi$ to the origin. It was proved in
\cite{KMY} that there are no such domains with smooth boundary if $n$ is even, and  if $n$ is odd then the only polynomially integrable domains with smooth boundary are ellipsoids.  In this article, we modify the notion of polynomial integrability for even $n$ and consider bodies for which the sectional volume function is a polynomial up to a factor which is the square root of a quadratic polynomial, or, equivalently, the Hilbert transform of this function is a polynomial. We prove that ellipsoids in even dimensions are the only convex infinitely smooth bodies satisfying this property.

\end{abstract}
\maketitle



\section{Formulation of the problem and the main result}

The following notion was introduced in \cite{Ag1}.
\begin{definition} Let $K$ be a  bounded domain in $\mathbb R^n.$ Then $K$ is called {\it polynomially integrable} if the Radon transform of its
characteristic function
$$A_K(\xi,t)=R\chi_K(\xi,t)= \int\limits_{K \cap \{ x \cdot \xi=t \} } dx, \ \xi\in S^{n-1}, \ t \in \mathbb R,$$
is a polynomial in $t$:
$$A_K(\xi,t)=\sum\limits_{j=0}^N a_j(\xi)t^j$$
for $t$ such that the hyperplane $x \cdot \xi=t$ intersects $K.$
\end{definition}

Polynomially integrable domains with $C^{\infty}$ boundary were fully characterized in \cite{Ag1}, \cite{KMY}. First, there are no
such domains in $\mathbb R^n$ with even $n$. Secondly, if $n$ is odd then ellipsoidal domains exhaust the class of such domains:
\begin{theorem}[\cite{KMY}] \label{T:KMY} Let $K$ be a bounded domain in $\mathbb R^n$ with an infinitely smooth boundary $\partial K.$ If $K$ is polynomially integrable
then $n$ is odd and $K$ is an ellipsoid.
\end{theorem}
\begin{remark} Theorem \ref{T:KMY} was formulated in \cite{KMY} for convex bodies $K.$ However, it was proved in \cite{Ag1} that
polynomially integrable domains in $\mathbb R^{2k+1},$ with smooth boundary, are necessarily convex and thus the convexity assumption in Theorem \ref{T:KMY} is superfluous. Also, when $K$ is a convex body, the function $A_K(\xi, t)$ is
continuous with respect to $\xi$, which implies that the coefficients
$a_j(\xi)$ are a priori continuous functions on the unit sphere.
\end{remark}

In this article, we introduce an analog of polynomial integrability in even dimensions.  First of all, there are no polynomially integrable convex domains with smooth boundary in even-dimensional spaces. This was proved in \cite{Ag1} and \cite{KMY} using different arguments.
The proof in \cite{Ag1} relies on the behavior of the sectional volume function $A_K(\xi,t)$ near
 the tangent plane $T_a(\partial K)=\{ x \cdot \xi=t_0 \} $  to the boundary at a point $a \in \partial K$ (see Lemma \ref{L:L1}). The argument is as follows:
 for almost all normal vectors $\xi \in S^{n-1}$  near the tangent plane we have $ A_K(\xi,t)=const \ (t-t_0)^{\frac{n-1}{2}}(1+o(1)), \ t \to t_0.$
 If $n$ is even then $\frac{n-1}{2}$ is half-integer and therefore $A_K(\xi,t)$ cannot be a polynomial in $t.$

 In order to formulate the main result of the article we need some notations.
The support functions of a compact convex body $K \subset \mathbb R^n$ are defined by

\begin{align}
&  h_{K}^{+}(\xi)=h_{K}(\xi)=\max_{ x \in K} x \cdot\xi,\\
&  h_{K}^{-}(\xi)=\min_{ x \in K} x \cdot\xi.
\end{align}
where $\xi$ belongs to the unit sphere  $S^{n-1}$ in $\mathbb R^n.$
Clearly, $h_{K}^{-}(\xi)=-h_K^{+}(-\xi)$ and a hyperplane $ \{x \cdot \xi =t \}$ meets {the interior} of  $K$ if and only if $t \in I_{\xi}:= \big(h_K^{-}(\xi), h_K^{+}(\xi) \big).$

Denote by $\mathcal{H}$ the Hilbert transform
\begin{equation}\label{E:H}
\mathcal{H} f(t)= \frac{1}{\pi} p.v. \int\limits_{\mathbb R}  \frac{f(s)}{t-s} ds
\end{equation}
of a continuous function $f$ with sufficiently fast decay at infinity.

The main result of this article is as follows.
\begin{theorem} \label{T:Main} Let $n$ be an even positive integer. Let $K$ be a bounded convex domain in $\mathbb R^n$ with $C^{\infty}$ boundary $\partial K.$
The following are equivalent:
\begin{enumerate}[(i)]
\item The sectional volume function $A_K(\xi,t)$ has for $t \in I_{\xi}$ the form
$$A_K(\xi,t)=\sqrt{q(\xi,t)} P(\xi,t),$$
where $P(\xi,t), \ q(\xi,t)$ are {continuous in $\xi$} and polynomials in $t$ with $\deg q(\xi,\cdot) =2 ; \ q(\xi,t) > 0, t \in I_{\xi}.$
\item The sectional volume function $A(\xi,t)$ has for $t \in I_{\xi}$ the form
$$A_K(\xi,t)=\frac{P(\xi,t)}{\sqrt{q(\xi,t)}},$$
where $P(\xi,t), \ q(\xi,t)$ are as in (i).
\item The Hilbert transform $\mathcal{H}A_K(\xi,t)$ is a polynomial with respect to   $t \in I_{\xi}$ for each  $\xi\in S^{n-1}$, i.e.,
$$\mathcal{H}A_K(\xi,t)=\sum\limits_{j=0}^N b_j(\xi)t^j,$$
where   $N$ is an integer and  $b_j$ are some (a priori continuous) functions on the unit sphere.
\item $K$ is an ellipsoid.
\end{enumerate}

\end{theorem}

\section{Proof of Theorem \ref{T:Main}. Equivalence of conditions $(i), (ii),(iii)$} \label{S2}

We start with some preliminary facts.

\subsection{Boundary behavior of the sectional volume function}

In the case where $K$ is an ellipsoid, the support function $h_K(\xi)$ is the restriction to the unit sphere $|\xi|=1$ of the square root of
a quadratic polynomial. In fact, for the ellipsoid  $E$  written in suitable coordinates in the standard form
\[
  E=\left\{  \sum_{j=1}^{n}\frac{x_{j}^{2}}{a_{j}^{2}} \le 1\right\}
\]
we have
\[
h_{E}(\xi)=\sqrt{\sum_{j=1}^{n}a_{j}^{2}\xi_{j}^{2}}.
\]
Also one can check  that
\begin{equation}\label{AE}
A_E(\xi, t) = C_n \mathrm{Vol}_n(E)\, h_E^{-n}(\xi)\left(h_E^2(\xi)-t^2\right)^{(n-1)/2},\end{equation} for a certain constant $C_n$, and all $\xi$ and $t$ such that $x\cdot \xi =t$ intersects $E$. {It follows that if $n$ is odd  then $A_K(\xi,t)$ is a polynomial in $t$ and if $n$ is even then
$A_K(\xi,t)$ has the form $(ii)$ in Theorem \ref{T:Main} with $q(\xi,t)=h_E^2(\xi)-t^2.$}

A hyperplane $\{x\cdot\xi=t\}$ meets the domain $K$ if and only if
$t \in I_{\xi}=[h_K^{-}(\xi), \  h_{K}^+(\xi)]$ and the end points $t=h^{\pm}_K (\xi)$ of the segment $I_{\xi}$ correspond to the tangent hyperplanes
$$T_{a^{\pm}}(\partial K)=\{ x \cdot \xi=h_K^{\pm}(\xi)\} $$ at
the points $a^{\pm}\in\partial K$ such that the exterior unit normal vectors $\nu_{\partial K}(a^{\pm})$
are correspondingly $\nu_{\partial K}(a^{\pm})=\pm\xi.$

The behavior of the sectional volume function $A_K(\xi,t)$ near the tangent planes is given by the following Lemma (see \cite[Ch. 1, Section 1.7]{GGV},  \cite[Section 3, p.7]{Ag1},  \cite[Lemma 2.2]{AgK}).

\begin{lemma} \label{L:asymp}
\label{L:zeros} There is a  dense subset $\Sigma \subset S^{n-1},$ such that
the following asymptotic relation with respect to $t$ holds with some nonzero coefiicients
$c^{\pm}(\xi),$ non-vanishing for  $\xi  \in \pm \Sigma$, correspondingly:
\begin{align}
&A_K(\xi,t)=c^{+}(\xi)(h_{K}^{+}(\xi)-t) ^{\frac{n-1}{2}}(1+\mathit{o}(1)),\ t\rightarrow h_{K}^{+}(\xi)-0, \ \xi \in \Sigma, \\
&A_K(\xi,t)=c^{-}(\xi)(t-h_{K}^{-}(\xi))^{\frac{n-1}{2}}(1+\mathit{o}(1)),\ t \rightarrow h_{K}^{-}(\xi) + 0, \ \xi \in -\Sigma.
\end{align}


\end{lemma}

\begin{proof}
We will use the notation $\Gamma=\partial K.$ Then $\Gamma$ is an
infinitely differentiable closed hypersurface. Let $\kappa_{\Gamma} (a), \ a \in\Gamma$ be the
Gaussian curvature of $\Gamma$ at the point $a.$

Denote by $\gamma$ the Gauss mapping
\[
\gamma:\Gamma\ni a\rightarrow\nu_{\Gamma}(a)\in S^{n-1},
\]
which maps a point $a\in\Gamma$ to the exterior unit normal vector
$\gamma(a)=\nu_{\Gamma}(a)$ to $\Gamma$ at the point $a.$  The mapping $\gamma$ is differentiable and the
Gaussian curvature $\kappa_{\Gamma}(a)$ is equal to the  Jacobian determinant
$\kappa_{\Gamma}(a)=J_{\gamma}(a)$ of $\gamma$ at the point $a.$ Therefore, the points
$a$ with $\kappa_{\gamma}(a)\neq0$ (\textit{non-degenerate points}) constitute
the set $\mathrm{Reg}_{\gamma}$ of regular points of the mapping $\gamma,$
while the set of points $a$ of zero Gaussian curvature coincides with the
critical set $\mathrm{Crit}_{\gamma}.$

By Sard's theorem,
the set $\gamma(\mathrm{Crit}_{\gamma})$ has the Lebesgue
measure zero on $S^{n-1}$, while the set
$$\Sigma =S^{n-1}  \setminus \gamma(\mathrm{Crit}_{\gamma})$$
of regular values is a dense subset of $S^{n-1}.$ It consists
of the unit vectors  $\xi$ such that any point $ a \in \Gamma$ with
$\nu_{\Gamma}(a) =\xi $ is non-degenerate.

Let $\xi\in \Sigma$ and let $a \in \Gamma$ be such that $a \cdot \xi  = h_K^+(\xi).$  The hyperplane $x \cdot \xi=h_K^{+}(a)$ is tangent to
$\Gamma$ and hence the external normal unit vector $\gamma(a)=\nu_{\Gamma}(a)=\xi.$ Since $\xi$ is a regular value of $\gamma,$ the point $a$ is non-degenerate, i.e., $\kappa_{\Gamma}(a) \neq 0.$
Applying a suitable translation and an orthogonal
transformation, we can make $a=0$ and $\xi=(0,\ldots,0,1).$ Then the
tangent plane $T_{a}(\Gamma)$ is the coordinate plane $x_{n}=0$ and the domain
$K$ is contained in the half-space $x_{n}\le 0.$ In this case $h_K^{+}(\xi)=0.$ Moreover, after performing a suitable non-degenerate linear
transformation we can make the equation of $\Gamma,$ near $a=0,$ to be:%

\begin{equation}
\label{E:x_n=}x_{n}=-\frac{1}{2}\left(  c_{1}x_{1}^{2}+\cdots+c_{n-1}%
x_{n-1}^{2}\right)  +\mathit{o}\left(  |x^{\prime}|^{2}\right)  ,\ (x_{1}%
,\ldots,x_{n-1})=x^{\prime}\rightarrow0.
\end{equation}
The new axes $x_{j}$, {$j=1,\ldots,n-1$}, are the   directions of the vectors of
principal curvatures and the coefficients $c_{j}$ are the values of the
principal curvatures at the point $a=0\in\Gamma.$ The Gaussian curvature at
$a=0$ is $\kappa_{\Gamma}(0)=c_{1}\cdots c_{n-1}$. All the applied
transformations preserve regular points, hence $\kappa_{\Gamma}(0)\neq 0$.
Therefore, none of $c_{j}$'s are equal to zero, and, since $c_{j}\geq 0$ due to the
convexity of $\Gamma,$ we have $c_{j}>0$ for all $j.$

After the above transformations we have $\xi=(0, \ldots, 0,1),$ so the
hyperplane $x\cdot\xi=t$ is now given by the equation $x_{n}=t,$ with
$t<0.$ The main term of $\mathrm{Vol}_{n-1}(K \cap\{x_{n}=t\})$ near
$t=0$ is determined by the main term of the expansion (\ref{E:x_n=}), i.e., by
the volume of the ellipsoid $-2t=c_{1}x_{1}^{2}+\cdots+c_{n-1}x_{n-1}^{2},$
which is equal to $c(-t)^{\frac{n-1}{2}}$, where
$
 c=\frac{(2\pi)^{\frac{n-1}{2}}}{\Gamma(\frac{n+1}%
{2})\sqrt{\kappa_{\Gamma}(a)}}.
$

Thus, for the specific choice $a=0$ and $\xi=(0,\ldots,0,1),$ we have the
following asymptotic formula:
$$A_K(\xi,t)=\mathrm{Vol}_{n-1}(K\cap
\{x_{n}=t\})=c \ (-t)^{\frac{n-1}{2}}+o(|t|^{\frac{n-1}{2}}),t\rightarrow-0,
$$
near $(\xi,t_0)$ with $\xi=(0,0,...,0,1)$ and $t_0= h_K^{+}(\xi)=0.$ Performing the inverse affine transformation, we obtain the first asymptotic formula in
Lemma \ref{L:asymp},  with some new nonzero
constant $c^{+}$ depending, of course, on $\xi$.

{The second asymptotic relation follows from the first one and from the relations $h_K^{+}(-\xi)=-h_K^{-}(\xi), \ A_K(-\xi, -t) =A_K(\xi,t).$}

\end{proof}
\smallbreak

Lemma \ref{L:asymp} implies an explicit form of the quadratic polynomial $q$ in conditions $(i), \ (ii)$ of Theorem \ref{T:Main}, as follows:
\begin{lemma} \label{L:L1}  Let $n \geq 2$ be an even integer, and let $K$ be a bounded convex body in $\mathbb R^n $ with $C^{\infty}$ boundary $\partial K.$
Let $q(\xi,t)$ be a quadratic polynomial of $t$ in  condition $(i)$ or in condition $(ii)$ of Theorem \ref{T:Main}.
Then $q(\xi,t)=q_0(\xi) \big(h_K^+(\xi)-t \big) \big (t-h_K^{-}(\xi) \big).$
\end{lemma}
{\bf Proof.}
Let $n=2m.$ Let us start with the case $(ii):$
\begin{equation}\label{E:2}\sqrt{q(\xi,t)} A_K(\xi,t)=P(\xi,t),
\end{equation}
where $P$ is a polynomial in $t.$

By Lemma \ref{L:asymp},  there is a dense set $\Sigma \in S^{n-1}$ such that the function $t \to A_K^2 (\xi,t)$ vanishes at the points
$h_K^{\pm}(\xi)$ (when $\xi \in \pm \Sigma,$ respectively) to the order exactly $2\frac{n-1}{2}=2m-1.$
Therefore, for any  $\xi \in \Sigma$ we have
$$P^2(\xi,t)=q(\xi,t)A^2(\xi,t)=q(\xi,t) \big ( h_K^{+}(\xi)-t \big)^{2m-1}P_0(\xi,t),$$
where $P_0(\xi,t)$ is another polynomial with respect to $t$ and $P_0 \big( \xi, h_K^{+}(\xi) \big) \neq 0.$
Then $P^2(\xi,t)$ has zero at $t=h_K^{+}(\xi),$ of even multiplicity. Comparing the multiplicities at both sides of the equality , we obtain  $q(\xi, h_K^{+}(\xi) )=0.$ Since $\Sigma$ is a dense subset of $S^{n-1}$ and $q(\xi,t), h_K^+(\xi)$ are continuous with respect to $\xi,$ this is true for all $\xi \in S^{n-1}.$

A similar argument using the expansion from Lemma \ref{L:asymp} at the point $h_K^{-}(\xi)$ implies that $q(\xi, h_K^{-}(\xi))=0, \xi \in S^{n-1}.$ Since $q(\xi,t)$ is a quadratic polynomial in $t,$  the needed presentation for $q(\xi,t)$ follows.

The case (i) easily reduces to $(ii)$. Indeed, if  $A_K(\xi,t)=\sqrt{q(\xi,t)}P(\xi,t)$ then $\sqrt{q(\xi,t)} A_K(\xi,t)=q(\xi,t)P(\xi,t)$ and this is the case $(ii)$ because in the right hand side we have  a polynomial in $t$. The lemma is proved.

\qed

\subsection{Functions with polynomial Hilbert transform on a finite interval}

We will need  some facts about the Hilbert transform (\ref{E:H}). This transform is originally  defined on continuous functions with sufficiently fast decay at infinity, but can be extended to less decaying functions and also to distributions.
The Hilbert transform $\mathcal{H}$  is self-invertible; more precisely $\mathcal{H(H}F)=-F.$ We have the
following intertwining relation between the transform $\mathcal H$ and the operator of multiplication by the independent variable (see \cite[Section 4.7]{King}) :
\begin{equation}
 \mathcal{H}(s\varphi(s)) (t)=t   \mathcal{H}%
\varphi  (t)-\frac{1}{\pi}\int\limits_{\mathbb{R}}\varphi(s) \, ds\text{.}
\label{E:intertwine}%
\end{equation}
Let $\chi_{\lbrack a,b]}(s)$ be the characteristic function of the interval
$[a,b].$ The Hilbert transform of the function $\chi_{\lbrack-1,1]}%
(s)\sqrt{(1-s)(1+s)}$ is well-known (see \cite[formula 11.343]{King}):%
\[
\mathcal{H} \left(  \chi_{\lbrack-1,1]}(s)\sqrt{(1-s)(1+s)} \right) (t)=t,\quad t\in\lbrack -1,1].
\]
By a linear change of variables one obtains the Hilbert transform of
$\chi_{\lbrack a,b]}(s)\sqrt{(b-s)(s-a)}$:%

\begin{equation}\label{E:Hilbert_arch}
  \mathcal{H}\left(  \chi_{\lbrack a,b]}(s)\sqrt{(b-s)(s-a)}\right)
  (t)=t-\frac{b+a}{2},\quad t\in\lbrack a,b]. %
\end{equation}
We will also make use of the inversion formula for the Hilbert transform on a finite interval ({\bf finite  Hilbert
transform}).
Namely, if a continuous function $F(t)$ is
supported on an interval $[a,b]$, then $F$ can be recovered from the knowledge of the values of its Hilbert transform only on $[a,b].$ The corresponding inversion formula looks as follows (see, e.g., \cite{you2006explicit}) :
\begin{equation}\label{E:finite_Hilbert}
\sqrt{(b-t)(t-a)} F(t)= - \mathcal{H} \bigg( \chi_{[a,b]}(s) \mathcal{H}F (s)\sqrt{(b-s)(s-a)}  \bigg ) (t) +\frac{1}{\pi}\int\limits_{a}^{b}F(s) \ ds , \ t\in  [a,b].
\end{equation}

\begin{lemma} \label{L:Hilbert} Let $[a,b]$ be a segment on the real line and let $F$ be a continuous function on the real line, supported in the segment $[a,b].$
Then the  following properties are equivalent:
\begin{enumerate}[(a)]
\item The function
$\sqrt{(b-t)(t-a)} F(t)$ is a polynomial on the interval  $t \in (a,b)$.
\item The function $\frac{F(t)}{\sqrt{(b-t)(t-a)}}$ is a polynomial on the interval $t \in (a,b)$.
\item The Hilbert transform
$\mathcal{H}F (t)$ is a polynomial on the interval $t \in (a,b).$
\end{enumerate}
\end{lemma}

\begin{proof}

$ (a)\Leftrightarrow (b)$

If (a) holds then $\sqrt{(b-t)(t-a)} F(t)=Q(t), \ t \in (a,b),$ where $Q$ is a polynomial. Since $F(t)$ is continuous at $t=a$ and $t=b$ we have $Q(a)=Q(b)=0$, and by Bezout's theorem $Q(t)=(b-t)(t-a)Q_1(t),$
where $Q_1$ is another polynomial. Thus,$\sqrt{(b-t)(t-a)} F(t)=(b-t)(t-a)Q_1(t),$ $ t \in (a,b)$ and hence $F(t)=\sqrt{(b-t)(t-a)}Q_1(t),$ \ $t \in (a,b),$ which is exactly condition (b).

Conversely, if (b) holds then  $F(t)= \sqrt{(b-t)(t-a)}Q(t)$, $ t \in (a,b), \ Q$ is a polynomial.  Multiplying both sides by $\sqrt{( b-t)(t-a)}$ leads to $\sqrt{(b-t)(t-a)}F(t)=(b-t)(a-t)Q(t)$, \ $ t \in (a,b),$ and therefore (a) holds.

$(c) \Rightarrow (a)$

Suppose that $\mathcal {H} F(t)=P(t), \ t \in (a,b),$ where $P$ is a polynomial.

Then inversion formula  (\ref{E:finite_Hilbert}) reads as
\begin{equation}\label{E:f1}
 \sqrt{(b-t)(t-a)} F(t)= - \mathcal{H} \bigg(\chi_{[a,b]}(s)P(s)\sqrt{(b-s)(s-a)} \bigg ) (t) +\frac{1}{\pi}\int\limits_{a}^{b}F(s) \ ds, \ t\in  [a,b],
\end{equation}
and therefore, to prove (a), it suffices to prove that  the right hand side is a polynomial on the interval $(a,b).$ In turn, it suffices to check this only for monomials
$P(s)=s^k.$

Thus, we need to prove that the Hilbert transform of the function $\chi_{[a,b]}(s)s^k\sqrt{ (b-s)(s-a)}$ is a polynomial.

It is true for $k=0$ because  identity (\ref{E:Hilbert_arch}) yields
$$\mathcal{H} \bigg( \chi_{[a,b]} (s) \sqrt { (b-s)(s-a)} \bigg)(t) =t+c_0, \  t \in [a,b],$$
for a certain constant $c_0.$  For $k>0$ formula (\ref{E:finite_Hilbert}) leads to
$$\mathcal{H} \bigg( \chi_{[a,b]}(s)s^k \sqrt{ (b-s)(s-a)} \bigg )(t)= t \bigg[  \mathcal{H} \bigg( \chi_{[a,b]}(s)s^{k-1}\sqrt{ (b-s)(s-a)} \bigg)(t)+c_k \bigg ],$$
where $c_k$ is a constant.
Thus, by induction, the above two equalities imply that
$$\mathcal{H} \bigg( \chi_{[a,b]}(s)s^k\sqrt{ (b-s)(s-a)} \bigg)(t)$$ is a polynomial
of degree $k+1.$ Thus, the right hand side in (\ref{E:f1}) is a polynomial when $P(s)$ is a monomial of an arbitrary degree and hence this is true for any polynomial $P$ which proves (a).

$(b) \Rightarrow (c)$

 If (b) is fulfilled then $F(t)=\sqrt{(b-t)(t-a)} Q(t), \ t \in (a,b)$ for some polynomial $Q(t).$
 Then inversion formula (\ref{E:finite_Hilbert}) for the finite Hilbert transform on $[a,b]$ can be written as
 $$(b-t)(t-a)Q(t) =- \mathcal{H} \bigg (\chi_{[a,b]}(s) \mathcal{H}F(s)\sqrt{(b-s)(s-a)} \bigg) (t) +c_1, \ t \in [a,b],$$
 where $c_1$ is a constant.

 Denote for convenience $G(s)=\chi_{[a,b]}(s) \mathcal{H}F(s)\sqrt{(b-s)(s-a)}.$ Then for $ t \in [a,b]$ we have
 $$\mathcal{H} G(t)=Q_1(t),$$
 where $Q_1(t)=-(b-t)(t-a)Q(t)+c_1.$
Again, inversion formula (\ref{E:finite_Hilbert}) yields:
$$G(t)\sqrt{(b-t)(t-a)}=-\mathcal {H} \bigg (\chi_{[a,b]}
(s)\sqrt{(b-s)(s-a)} Q_1(s) \bigg)(t) +c_2,$$
with some constant $c_2.$  We have just proven that the expression in the right hand side is a polynomial on $t \in [a,b].$

Substituting the expression for $G$ we arrive at
$$(b-t)(t-a)\mathcal{H}F(t)= P(t), \ t \in (a,b),$$
where $P(t)$ is a polynomial.
 {Since $F(t)$ is bounded on the real line, $|F(t)| \leq C,$ and supported in $[a,b],$ its  Hilbert transform satisfies
$|\mathcal{H}F(t)| \leq \frac{C}{\pi} \  \ln \frac{b-t}{t-a} , \ t \in (a,b).$
Hence  the limits, as $t \to a, \ t \to b,$ of the left hand side of the above equality are equal to zero.}  This implies $P(a)=P(b)=0$ and
hence $P(t)=(b-t)(a-t)P_1(t),$ where $P_1$ is a polynomial. Then  $\mathcal{H}F(t)=P_1(t), \ t \in (a,b)$ and property (c) is proved.
Thus, we have proven that the properties (a), (b), (c) are equivalent. The Lemma is proved.

\end{proof}

\subsection{Equivalence of conditions $(i), (ii),(iii)$ (end of the proof) }

The equivalence of conditions $(i),(ii)$ and $(iii)$ of Theorem \ref{T:Main} follows immediately from Lemma \ref{L:L1} and also from Lemma \ref{L:Hilbert} applied
to $F(t)=A_K(\xi,t), \ a=h_K^{-}(\xi), \ b =h_K^{+}(\xi).$  {Indeed, Lemma \ref{L:L1} gives an explicit form of the quadratic polynomial $q(\xi,t)$ in $(i),(ii)$ and says that conditions $(i),(ii),(iii)$ for $A_K(\xi,t)$ read as conditions $(a),(b),(c),$ respectively, for the function $F(t)$ in Lemma \ref{L:Hilbert}. The latter lemma claims that conditions $(a),(b),(c)$ are equivalent and therefore conditions $(i),(ii),(iii)$ are equivalent, too.}
The proof is completed.

\section{Proof of Theorem \ref{T:Main}. Equivalence of conditions $(iii)$ and $(iv)$}

Let us first show that  $(iv)$ implies $(iii)$.  {Suppose that $(iv)$ holds, i.e., $K$ is an ellipsoid. Applying a translation, if needed, we may assume that the center of the ellipsoid is at the origin, and therefore its section function  $A_K(\xi,t) $  is given by (\ref{AE}). Since $n$ is even,  $A_K(\xi,t)$ satisfies $(ii)$ with $q(\xi,t)=h_K^2(\xi)-t^2.$ It suffices to notice that, as we have proven in the previous section, conditions $(ii)$ and $(iii)$ are equivalent.}

We will now prove that $(iii)$  implies  $(iv)$. Before we start, let us  outline the plan of the proof. Let $K$ be a convex body satisfying $(iii)$. Without loss of generality we may assume that the origin is an interior point of $K$.
Since $\mathcal{H}A_K(\xi,t)$ is a polynomial in $t$ of degree at most $N$, the derivatives of  $\mathcal{H}A_K(\xi,t)$ with respect to $t$ of orders greater than $N$ at $t=0$ are equal to zero. In order to find derivatives of $\mathcal{H}A_K(\xi,t)$ at $t=0$, we will compute its fractional derivatives. The reader is referred to \cite[Section 2.6]{K} for more details about such techniques.
 The next step is to express fractional derivatives of $\mathcal{H}A_K(\xi,t)$ at zero  in terms of the Fourier transform of expressions involving powers of the Minkowski functional of $K$. Recall that the latter is defined by
$$\|x\|_K=\min\{a\ge 0: x\in aK\},\qquad x\in \mathbb R^n.$$
Since ordinary derivatives are obtained by computing fractional derivatives at positive integers, we will get the condition that the Fourier transform of  $ \|-x\|_K^{-n+1+m} +(-1)^{m+1}   \|x\|_K^{-n+1+m}$ must be concentrated at the origin for large enough integers $m$. This implies that  $ \|-x\|_K^{-n+1+m} +(-1)^{m+1}   \|x\|_K^{-n+1+m}$ must be a homogeneous polynomial of $x$.  An  algebraic result from   \cite{KMY} then implies that $K$ must be an ellipsoid in even dimensions.

Now we will provide details of the above plan.
Let us write the Hilbert transform of $A_K(\xi,t) $ as follows
$$\mathcal{H}A_K(\xi,t) = \frac{1}{\pi} \int_0^\infty \frac{A_K(\xi,t-z)-A_K(\xi,t+z)}{z} \, dz.$$

Let $q$ be a complex number such that $-1<\Re q <0$. Consider the fractional derivative of order $q$ at $t=0$ of the function $\mathcal{H}A_K(\xi,t)$.

$$(\mathcal{H}A_K)^{(q)}(\xi,0)=\frac{1}{\Gamma(-q)} \int_0^\infty t^{-1-q} \, \mathcal{H}A_K(\xi, -t)\, dt.$$

Let us briefly explain why the last integral converges. $\mathcal{H}A_K(\xi,t)$ is a continuous function of $t$ on $\mathbb R$ except possibly at the points $t=h_K^+(\xi)$ and $t=h_K^-(\xi)$, where in the worst case it behaves as  $\ln |h_K^+(\xi) - t|$ and $\ln |h_K^-(\xi) - t|$ respectively.
Additionally, as $t\to \pm \infty$ it behaves as $1/t$.

Writing $\mathcal{H}A_K$ as follows:
$$\mathcal{H}A_K(\xi,t) = \lim_{\epsilon\to 0^+} \frac1{\pi} \int_0^\infty \frac{A_K(\xi,t-z)-A_K(\xi,t+z)}{z^{1+\epsilon}} \, dz,$$
and using the dominated convergence theorem and Fubini's theorem we get
\begin{align*} (\mathcal{H}A_K)^{(q)}(\xi,0)& = \lim_{\epsilon\to 0^+}\frac{1}{\pi \Gamma(-q)} \int_0^\infty t^{-1-q} \int_0^\infty \frac{A_K(\xi,-t-z)-A_K(\xi,-t+z)}{z^{1+\epsilon}} \, dz\, dt
\\
&= \lim_{\epsilon\to 0^+}\frac{1}{\pi \Gamma(-q)} \int_0^\infty \frac1{z^{1+\epsilon}} \int_0^\infty t^{-1-q} \left(A_K(\xi,-t-z)-A_K(\xi,-t+z)\right) \, dt\, dz\\
&
= \lim_{\epsilon\to 0^+}
\frac{1}{\pi \Gamma(-q)} \int_0^\infty \frac1{z^{1+\epsilon}} \int_{\mathbb R^n} \Big(\left(-z  -   x\cdot\xi  \right)_+^{-1-q}- \left(z  -x\cdot\xi  \right)_+^{-1-q}\Big) \chi(\|x\|_K)\, dx\, dz.
\end{align*}
Here and below we use the following notation. If $\Re \lambda >-1$, then $$t_+^\lambda =\begin{cases}0, & t\le 0,\\ t^\lambda, & t> 0 ,\end{cases} \qquad \mbox{and} \qquad t_-^\lambda =\begin{cases}|t|^\lambda, & t< 0,\\ 0, & t\ge 0. \end{cases}$$

Observe that $(\mathcal{H}A_K)^{(q)}(\xi,0)$ naturally extends to a homogeneous function of $\xi\in \mathbb R^n$ of degree $-1-q$, and we will consider its distributional Fourier transform with respect to $\xi$. Let $\phi$ be a Schwarz function. Then
\begin{align*}&\langle \left((\mathcal{H}A_K)^{(q)}(\cdot,0)\right)^\wedge, \phi\rangle = \langle  (\mathcal{H}A_K)^{(q)}(\xi,0) , \hat\phi(\xi) \rangle\\
&
=
\lim_{\epsilon\to 0^+}\frac{1}{\pi \Gamma(-q)} \int_0^\infty \frac1{z^{1+\epsilon}} \int_{\mathbb R^n} \chi(\|x\|_K) \int_{\mathbb R^n} \Big(\left(-z  -   x\cdot\xi  \right)_+^{-1-q}- \left(z  -   x\cdot\xi  \right)_+^{-1-q}\Big) \hat\phi(\xi)\, d\xi \, dx dz\\
&
=\lim_{\epsilon\to 0^+}
\frac{1}{\pi \Gamma(-q)} \int_0^\infty \frac1{z^{1+\epsilon}} \int_{\mathbb R^n} \chi(\|x\|_K) \int_{\mathbb R} \Big(\left(u-z   \right)_+^{-1-q}- \left(u+z \right)_+^{-1-q}\Big) \int_{  x\cdot\xi =-u}\hat\phi(\xi)\, d\xi\, du \, dx dz.
\end{align*}
The Fourier transform of $\left(u-z   \right)_+^{-1-q}- \left(u+z \right)_+^{-1-q}$ with respect to $u$ equals
\begin{multline*}i\Gamma(-q) \Big(e^{i(-1-q)\pi/2}s_+^q e^{-izs}-e^{i(1+q)\pi/2}s_-^q e^{-izs} - e^{i(-1-q)\pi/2}s_+^q e^{izs}+e^{i(1+q)\pi/2}s_-^q e^{izs} \Big)\\
= 2\sin(zs) \Gamma(-q) \Big(e^{i(-1-q)\pi/2}s_+^q  -e^{i(1+q)\pi/2}s_-^q  \Big);
\end{multline*}
see \cite[Ch. II, Sec. 2.3]{GS}.

Using the connection between the Radon transform and the Fourier transform, we get
\begin{align*} (2\pi)^{-n+1}&\langle \left((\mathcal{H}A_K)^{(q)}(\xi,0)\right)^\wedge, \phi\rangle\\
&=\lim_{\epsilon\to 0^+}
\frac{2 }{\pi } \int_0^\infty \frac1{z^{1+\epsilon}} \int_{\mathbb R^n} \chi(\|x\|_K)  \int_{\mathbb R}  \sin(zs)   \Big(e^{i(-1-q)\pi/2}s_+^q  -e^{i(1+q)\pi/2}s_-^q  \Big) \phi(-sx)  \, ds \, dx\, dz\\
&=\lim_{\epsilon\to 0^+}
\frac{2 }{\pi }   \int_{\mathbb R^n} \chi(\|x\|_K)  \int_{\mathbb R}  \int_0^\infty \frac{\sin(zs)}{z^{1+\epsilon}} \, dz  \Big(e^{i(-1-q)\pi/2}s_+^q  -e^{i(1+q)\pi/2}s_-^q  \Big) \phi(-sx)  \, ds \, dx.
\end{align*}
The latter use of the Fubini theorem explains why we passed from $1/z$ to $1/z^{1+\epsilon}$ earlier: the integral of $\frac{\sin(zs)}{z^{1+\epsilon}}$ is absolutely convergent, while the integral of $\frac{\sin(zs)}{z }$ is not.
To compute $\int_0^\infty \frac{\sin(zs)}{z^{1+\epsilon}} \, dz $
we can write it as
$\frac1{2i}\int_0^\infty \frac{e^{izs}-e^{-izs}}{z^{1+\epsilon}} \, dz $ and then repeat the calculations from \cite[Ch. II, Sec. 2.3]{GS} for the Fourier transform of $z_+^{-1-\epsilon}$.
As a result we get
\begin{align*}\int_0^\infty \frac{\sin(zs)}{z^{1+\epsilon}} \, dz & = \frac1{2i} \left( i e^{-i(1+\epsilon)\pi/2}\Gamma(-\epsilon)\left( s_+^\epsilon+e^{i\epsilon \pi}s_-^\epsilon -  s_-^\epsilon - e^{i\epsilon \pi}s_+^\epsilon \right)\right) \\
&= \frac1{2}   e^{-i(1+\epsilon)\pi/2}\Gamma(-\epsilon)\left( s_+^\epsilon -   s_-^\epsilon\right) \left(1-e^{i\epsilon \pi} \right) \\
&=    -\sin(\epsilon \pi/2)\Gamma(-\epsilon)  |s|^\epsilon \mbox{sgn}(s).
\end{align*}
Therefore, $$\lim_{\epsilon\to 0^+} \int_0^\infty \frac{\sin(zs)}{z^{1+\epsilon}} \, dz = \frac{\pi}{2} \mbox{sgn}(s),$$
and hence
\begin{align*} (2\pi)^{-n+1}&\langle \left((\mathcal{H}A_K)^{(q)}(\xi,0)\right)^\wedge, \phi\rangle\\
&=     \int_{\mathbb R^n} \chi(\|x\|_K)  \int_{\mathbb R}  \mbox{sgn}(s) \Big(e^{i(-1-q)\pi/2}s_+^q  -e^{i(1+q)\pi/2}s_-^q  \Big) \phi(-sx)  \, ds \, dx\\
&=
\int_{\mathbb R^n} \chi(\|x\|_K)  \int_{0}^\infty  \mbox{sgn}(s)   e^{i(-1-q)\pi/2}s^q    \, ds\, \phi(-sx)  \, dx \\
&\qquad-
\int_{\mathbb R^n} \chi(\|x\|_K)  \int_{-\infty}^0  \mbox{sgn}(s)     e^{i(1+q)\pi/2}|s|^q    \, ds \,\phi(-sx)  \, dx \\
&=
\int_{\mathbb R^n} \chi(\|-x\|_K)  \int_{0}^\infty   e^{i(-1-q)\pi/2}s^q    \, ds\, \phi(sx)  \, dx \\
&\qquad+ \int_{\mathbb R^n} \chi(\|x\|_K)  \int_0^{\infty}        e^{i(1+q)\pi/2}s^q    \, ds \,\phi(sx)  \, dx \\
& =   e^{-i(1+q)\pi/2}
\int_{S^{n-1}} \int_0^{\|-\theta\|_K}r^{n-1} \int_0^{\infty} s^q\phi(sr\theta)\, ds \,dr \, d\theta\\
&\qquad+
e^{i(1+q)\pi/2}\int_{S^{n-1}} \int_0^{\|\theta\|_K}r^{n-1} \int_0^{\infty} s^q\phi(sr\theta) \,ds\, dr\, d\theta\\
& =   e^{-i(1+q)\pi/2}
\int_{S^{n-1}} \int_0^{\|-\theta\|_K}r^{n-2-q} \int_0^{\infty} s^q\phi(s\theta)\, ds \,dr \, d\theta\\
&\qquad+
e^{i(1+q)\pi/2}\int_{S^{n-1}} \int_0^{\|\theta\|_K}r^{n-2-q} \int_0^{\infty} s^q\phi(s\theta) \,ds\, dr\, d\theta\\
& = \frac{1}{n-1-q} \int_{S^{n-1}}  \left( e^{-i(1+q)\pi/2} \|-\theta\|_K^{-n+1+q} + e^{i(1+q)\pi/2} \|\theta\|_K^{-n+1+q}\right)
\int_0^{\infty} s^q\phi(s\theta) \,ds\,   d\theta\\
& = \frac{1}{n-1-q} \int_{\mathbb R^n}  \left( e^{-i(1+q)\pi/2} \|-x\|_K^{-n+1+q} + e^{i(1+q)\pi/2} \|x\|_K^{-n+1+q}\right)
\phi(x) \,dx.
\end{align*}

Thus, we have shown that
$$  (2\pi)^{-n+1}\left((\mathcal{H}A_K)^{(q)}(\xi,0)\right)^\wedge (x)= \frac{1}{n-1-q}  \left( e^{-i(1+q)\pi/2} \|-x\|_K^{-n+1+q} + e^{i(1+q)\pi/2} \|x\|_K^{-n+1+q}\right),$$
that is
$$   (\mathcal{H}A_K)^{(q)}(\xi,0) = \frac{1}{2\pi(n-1-q)}  \left( e^{-i(1+q)\pi/2} \|-x\|_K^{-n+1+q} + e^{i(1+q)\pi/2} \|x\|_K^{-n+1+q}\right)^\wedge(\xi),$$
for all complex $q$ such that $-1<\Re q<0$. Using analytic continuation, we see that the formula is still valid for all
$q\in \mathbb C\setminus\{n-1\}.$

Since  $\mathcal{H}A_K (\xi,t)$ is a polynomial of $t$ of degree at most $N$, we have  $(\mathcal{H}A_K)^{(m)}(\xi,0)=0$ for all $\xi\in \mathbb R^n\setminus\{0\}$ and all natural $m>\max\{N,n-1\}$.
This means that
$$ \left( e^{-i(1+m)\pi/2} \|-x\|_K^{-n+1+m} + e^{i(1+m)\pi/2} \|x\|_K^{-n+1+m}\right)^\wedge(\xi)$$
is a linear combination of derivatives of the delta function supported at the origin.
Thus, $$e^{-i(1+m)\pi/2} \|-x\|_K^{-n+1+m} + e^{i(1+m)\pi/2} \|x\|_K^{-n+1+m}$$ is a polynomial of $x$.

When $m$ is odd,   we get $$e^{-i(1+m)\pi/2} =e^{i(1+m)\pi/2},$$ and hence
$$  \|-x\|_K^{-n+1+m} +   \|x\|_K^{-n+1+m}$$ is a polynomial when $m$ is odd.

Similarly, when $m$ is even,  we get $$e^{-i(1+m)\pi/2} =-e^{i(1+m)\pi/2},$$ and thus
$$  \|-x\|_K^{-n+1+m} -  \|x\|_K^{-n+1+m}$$ is a polynomial when $m$ is even.

Now we apply the same reasoning as in  \cite[Theorem 3.7]{KMY} to show that $K$ is an ellipsoid if $n$ is even.

Finally, let us remark that  bodies with polynomial $\mathcal{H}A_K(\xi,t)$ do not exist in odd dimensions. This follows from the fact that  the function
$$  \|-x\|_K^{-n+1+m} +   \|x\|_K^{-n+1+m}$$ is an even function, but at the same time,  it has to be a polynomial of an odd degree $-n+m+1$, if $n$ and $m$ are both odd. The only polynomial that is both odd and even is the zero polynomial.

\bigskip
{\bf Acknowledgment.} 
The work on this paper began at BIRS where the authors participated in  the 2022  Research in Teams Program ``Algebraically integrable
domains". The authors are grateful to BIRS and its staff for their hospitality and excellent research conditions.


\begin{thebibliography}{99}                                                                                               %


\bibitem{Ag1}M.~Agranovsky, {\em On polynomially
integrable domains in Euclidean spaces},  In: Trends in Mathematics; Complex Analysis and Dynamical Systems.
New Trends and Open Problems, Birkhauser, 2018, pp. 1--21.

 

\bibitem{AgK} M.~Agranovsky and L.~Kunyansky, {\em On exactness of the universal backprojection formula for the spherical means Radon transform,} arXiv:2207.08262

\bibitem {GGV}I.~M.~Gelfand, M.~I.~Graev, and N.~I.~Vilenkin, {\em Generalized Functions. Vol. 5. Integral Geometry and
Representation Theory},  Academic Press, New York-London, 1966.

\bibitem{GS} I.~M.~Gelfand,  G.~E.~Shilov,   {\em  Generalized Functions. Vol. 1. Properties and Operations,} Translated from the Russian by Eugene Saletan. Academic Press, New York-London, 1964.

 

\bibitem {King}F.W. King, {\em Hilbert Transforms: Volume 1},
volume~2, \newblock Cambridge University Press, 2009.

\bibitem{K}
{A.~Koldobsky}, {\em Fourier Analysis in Convex Geometry}, American Mathematical Society, Providence RI, 2005.

\bibitem{KMY} A.~Koldobsky, A.~S.~Merkurjev, and V.~Yaskin,
\newblock On polynomially integrable convex bodies.
\newblock {\em Advances in Mathematics}, 320, 876--886, 2017.



 

\bibitem {you2006explicit}J. You and G.L. Zeng, {\em Explicit finite
inverse Hilbert transforms},  Inverse Problems, 22, no. 3, L7--L10, 2006.
\end{thebibliography}
\end{document}